\documentclass[11pt]{article}

\usepackage[a4paper,margin=1in]{geometry}
\usepackage[T1]{fontenc}
\usepackage[utf8]{inputenc}
\usepackage{lmodern}
\usepackage{microtype}
\usepackage{amsmath,amssymb,amsthm,mathtools}
\usepackage{enumitem}
\usepackage{xcolor}
\usepackage{hyperref}
\usepackage[nameinlink,capitalise]{cleveref}
\usepackage[mathscr]{eucal}
\usepackage{graphicx}
\usepackage{tikz}
\usepackage{pgfplots}
\pgfplotsset{compat=1.18}
\usepackage{float}

\definecolor{Acol}{RGB}{184,45,45}
\definecolor{Bcol}{RGB}{225,132,25}
\definecolor{Ccol}{RGB}{35,112,177}
\definecolor{Dcol}{RGB}{35,145,88}

\hypersetup{
  colorlinks=true,
  linkcolor=blue,
  citecolor=blue,
  urlcolor=blue
}

\theoremstyle{plain}
\newtheorem{theorem}{Theorem}[section]
\newtheorem{proposition}[theorem]{Proposition}
\newtheorem{lemma}[theorem]{Lemma}
\newtheorem{corollary}[theorem]{Corollary}

\theoremstyle{definition}
\newtheorem{definition}[theorem]{Definition}

\newtheorem{example}[theorem]{Example}

\theoremstyle{remark}
\newtheorem{remark}[theorem]{Remark}

\newcommand{\RR}{\mathbb R}
\newcommand{\NN}{\mathbb N}

\newcommand{\DM}{\operatorname{DM}}

\newcommand{\down}{\downarrow}

\newcommand{\sss}[1]{{\scriptscriptstyle #1}}

\newcommand{\set}[2]{\ensuremath{\left\{ #1\,:\; #2\right\}}}

\newcommand{\Ocl}[1]{%
  \overline{#1}\,^{\raisebox{0.3ex}{$\scriptstyle O$}}%
}

\newcommand{\bigveeL}{\mathop{\bigvee\nolimits^{\scriptscriptstyle L}}}
\newcommand{\bigveeDML}{\mathop{\bigvee\nolimits^{\scriptscriptstyle \DM(L)}}}
\newcommand{\bigveev}[1]{\mathop{\bigvee\nolimits^{\scriptscriptstyle #1}}}

\begin{document}

\title{Regular Sublattices, \(O\)-Closed Ideals and\\
Dedekind--MacNeille Completions}

\author{Kevin Abela\\
\small		Department of Mathematics \\
\small	Faculty of Science \\
\small		University of Malta \\
\small	Msida MSD 2080, Malta\\
\small \texttt{kevin.abela.11@um.edu.mt}
\and
Emmanuel Chetcuti\\
	\small Department of Mathematics\\
	\small	Faculty of Science\\
	\small	University of Malta\\
	\small	Msida MSD 2080  Malta\\
\small \texttt{emanuel.chetcuti@um.edu.mt}}

	\date{\today}
	\maketitle

\begin{abstract}
We study regular sublattices of lattices that satisfy the Join-Infinite Distributive Law (JID) under
Dedekind--MacNeille completion. Using order convergence, we identify the
lattice of $O$-closed ideals with Crawley's completion, and compare it with the Dedekind--MacNeille completion. When $\DM(L)$ satisfies (JID), this comparison yields extension results for
lattice homomorphisms. In particular, if $Y$ is a regular sublattice of
a complete lattice $L$, and $L$ and $\DM(Y)$ satisfy (JID), then the canonical extension lattice-embeds $\DM(Y)$ into
$L$ with join-regular image. We also give a sharp finite-dimensional
obstruction: there exists an infinitely distributive regular sublattice
of $\mathbb R^3$ whose Dedekind--MacNeille completion is not distributive, and therefore cannot be lattice-embedded in $\RR^3$.   In contrast, we prove that the inclusion of
every join-regular sublattice of $\mathbb R^2$ with a least
and a greatest element extends canonically to a lattice
embedding of its MacNeille completion into $\mathbb R^2$. 
\end{abstract}

\noindent\textbf{2020 Mathematics Subject Classification.}
Primary 06B23; Secondary 06B10, 06D05.

\section{Introduction}

Regular inclusions of Archimedean vector lattices extend to
regular embeddings of their Dedekind completions
\cite[Theorem 2.10]{GaoTroitskyXanthos}.    We examine the
corresponding permanence problem for general lattices.
Connections with order and unbounded-order convergence are
discussed in \cite{GaoTroitskyXanthos,AbelaChetcutiPositivity} and many others;
related extension theorems for vector-lattice homomorphisms
appear in \cite[Theorem~3.1]{LuxembergSchep1979} and
\cite[Theorem~3.1]{Bernau1986}.

For lattices, the relevant completion is the
Dedekind--MacNeille completion \cite{MacNeille1937}; see
\cite{Birkhoff,DaveyPriestley,Gratzer} for standard accounts.
Its failure to preserve distributivity goes back to
Funayama \cite{Funayama1944} and Dilworth--McLaughlin
\cite{DilworthMcLaughlin}. Subsequent criteria and related
distributive extensions were studied in
\cite{Erne1982,Erne1991,BallPultrNormalCompletion}, while
\cite{BezhanishviliEtAl2025} compares MacNeille completion
with other completion procedures.

Our approach uses order convergence, which on a lattice
agrees with the restriction of order convergence from its
MacNeille completion
\cite[Theorem~3]{AbelaChetcutiWeber}. Related structural
applications appear in \cite{AvilesBilokopytovTroitsky}.
For every lattice $L$, the $O$-closed ideals coincide with
Crawley's complete ideals \cite{Crawley1962}, giving
$\mathfrak I_L=K(L)$. Under the join-infinite distributive
law \textup{(JID)}, we describe the associated
$\sigma_L$-closure explicitly and establish its compatibility
with join-regular inclusions and lattice homomorphisms.

We give direct proofs of the classical criteria for
\textup{(JID)} in $\DM(L)$
\cite[Corollary~3.2.2]{Cornish1974}
\cite[Theorem~13]{Niederle2005}, and express Cornish's
extension theorem \cite[Theorem~3.1]{Cornish1974} through
$\sigma$-closures. In particular, if $Y$ is join-regular in
a complete lattice $L$, and both $L$ and $\DM(Y)$ satisfy
\textup{(JID)}, the canonical extension embeds $\DM(Y)$ into
$L$ with join-regular image. A chain example shows that
meet-regularity need not follow, even when the original
inclusion is regular and both the ambient lattice and the
completion are complete chains.

Section~\ref{sec:counterexample} gives a concrete realization
of the classical distributivity obstruction: an infinitely
distributive regular sublattice of $\mathbb R^3$, with a least
and a greatest element, whose MacNeille completion is
non-modular and hence admits no lattice embedding into the
ambient space. In contrast, we prove that the inclusion of
every join-regular sublattice of $\mathbb R^2$ with a least
and a greatest element extends canonically to a lattice
embedding of its MacNeille completion into $\mathbb R^2$.
Thus three is the least dimension in which the MacNeille
completion of a regular sublattice of $\mathbb R^n$ with
a least and a greatest element can fail to be lattice embeddable in the ambient vector lattice $\RR^n$.

\section{Preliminaries on regularity and infinite distributivity}
\label{sec:prelim}

\begin{definition}[Regular sublattice]
Let \(Y\) be a sublattice of a lattice \(L\). We say that \(Y\) is
\emph{join-regular} in \(L\) if  whenever the supremum
of a nonempty subset of \(Y\) exists in $Y$, then it also exists in \(L\), with
the same value.  \emph{Meet-regular} is defined dually.  $Y$ is said to be \emph{regular} in $L$ if it is both join- and meet-regular.
\end{definition}

\begin{definition}[Infinite distributivity]
A lattice \(L\) satisfies the \emph{Join-Infinite Distributive Law} \textup{(JID)} if,
whenever \(\bigvee A\) exists,
\[
 x\wedge\bigvee A=\bigvee_{a\in A}(x\wedge a)
\]
for every \(x\in L\). Dually, \(L\) satisfies the \emph{Meet-Infinite
Distributive Law} (MID) if, whenever \(\bigwedge A\) exists,
\[
 x\vee\bigwedge A=\bigwedge_{a\in A}(x\vee a)
\]
for every \(x\in L\). We call \(L\) \emph{infinitely distributive} if it
satisfies both JID and MID.
\end{definition}

\begin{remark}
Every lattice-ordered group and every vector lattice is infinitely distributive. See \cite[Chapter~1]{LuxemburgZaanen} and \cite[Chapter~1]{AliprantisBurkinshaw}.
\end{remark}


\section{MacNeille completions and \(O\)-closed ideals}
\label{sec:macneille}

\subsection{Order convergence and \(O\)-closed ideals}

Let \(L\) be a lattice. Recall that a net
\((x_\gamma)_{\gamma\in\Gamma}\) is said to \emph{order converge} to
\(x\in L\), written $
x_\gamma\xrightarrow{O}x$,
if there exist a directed set \(M\subseteq L\) and a filtered set
\(N\subseteq L\) satisfying $\bigvee M=x=\bigwedge N$,
such that for every pair \((m,n)\in M\times N\), the net
\((x_\gamma)\) is eventually contained in the interval \([m,n]\).




For a subset \(A\subseteq L\), define
\[
A^O_1
=
\{x\in L:
\exists (a_\gamma)\subseteq A
\text{ with }
a_\gamma\xrightarrow{O}x\}.
\]
This set is called the \emph{first \(O\)-adherence} of \(A\). We say that
\(A\) is \emph{\(O\)-closed} if $A=A^O_1$.  Let $\Ocl{A}$ be the smallest $O$-closed subset of $L$ that contains $A$.  

The empty subset $\emptyset$ will be regarded as an ideal of lattice $L$
if and only if $L$ possesses no  bottom element. We denote by \([A]\) the
sublattice generated by \(A\), and by \(\mathscr H_L[A]\) the ideal
generated by \(A\). We write
\(A^{\sigma_L}\) for the smallest \(O\)-closed ideal of \(L\) containing
\(A\), and let \(\mathfrak I_L\)
denote the collection of all \(O\)-closed  ideals of \(L\).


There is a classical completion closely related to the construction above.
In \cite{Crawley1962} Crawley studied ideals that are closed under all joins of its elements that exist in the ambient lattice.  These ideals form a lattice $K(L)$, known as Crawley's completion; see also Cornish \cite{Cornish1974}.
The connection with \(O\)-closed ideals follows from the next result.

\begin{theorem}[{\cite[Proposition 15, Theorem 16]{AbelaChetcutiPositivity}}]\label{ThmIdealOclosure}
Let $A$ be a downset  of a lattice $ L$. 
\begin{enumerate}[label={\rm(\roman*)}]
\item Every $x \in  A_{1}^{O}$ is the supremum of an increasing net in $ A$.  Moreover, if $L$ satisfies \textup{(JID)} then $A_{1}^{O}$ is a downset.
\item If $L$ satisfies \textup{(JID)}, and $A$ is an ideal, then $ A_{1}^{O}$ is an \(O\)-closed ideal.
\end{enumerate}
\end{theorem}

Note that although the cited results assume  infinite distributivity,
the portions of their proofs establishing the assertions above use
only \textup{(JID)}.


For every lattice $L$, an ideal is $O$-closed if and only if it is
closed under all existing nonempty joins of its elements.
Indeed, $O$-closedness implies join-closedness by considering
the increasing net of nonempty finite joins. Conversely, if a
net in a complete ideal order converges to $x$, its lower
directed witness set is contained in the ideal and has supremum
$x$, so $x$ also belongs to the ideal. With our convention for
the empty ideal, this gives
\[
\mathfrak I_L=K(L)
\]
for every lattice $L$. Thus $\mathfrak I_L$ is Crawley's
completion, described through order convergence.
Under \textup{(JID)}, Theorem~\ref{ThmIdealOclosure} additionally
provides the one-step description of the closure of an ideal
used below.


\subsection{The $\sigma_L$-closure operator}

Using the identification $\mathfrak I_L=K(L)$, we study Crawley's
completion through the closure operator $A\mapsto A^{\sigma_L}$,
with particular attention to regular inclusions and lattice
homomorphisms.

We use \emph{Dedekind complete} to mean conditionally complete, and
\emph{complete} when arbitrary joins and meets, including those of the empty
set, exist.

\begin{proposition}\label{3.4}
Let $ L$ be a lattice satisfying \textup{(JID)}.
\begin{enumerate}[label={\rm(\roman*)}]
    \item $A^{\sigma_{{{ L}}}}=\bigl( \mathscr H_ L[A]\bigr)^O_1$ for every $A\subseteq  L$.
    \item If $A\subseteq  L$ is a sublattice, a point $x\in  L$ belongs to $A^{\sigma_ L}$ iff $x=\bigvee_{a\in A}^{L} x\wedge a$. 
\end{enumerate}
\end{proposition}
\begin{proof}
{\rm(i)}~Every \(O\)-closed ideal of $ L$ that contains $A$ must 
contain $\bigl( \mathscr H_ L[A]\bigr)^O_1$.  So it suffices to verify that  
$\bigl( \mathscr H_ L[A]\bigr)^O_1$ is an \(O\)-closed ideal of $ L$. 
This is guaranteed by {\rm(ii)} of Theorem \ref{ThmIdealOclosure}.

{\rm(ii)}~If $x\in L$ satisfies $x=\bigvee_{a\in A}^ L x\wedge a$, then $x$ belongs to the first 
O-adherence of $ \mathscr H_ L[A]$, and therefore $x\in A^{\sigma_ L}$.  Conversely, 
if $x\in A^{\sigma_ L}$, then $x\in \bigl( \mathscr H_ L[A]\bigr)^O_1$ by {\rm(i)}, so by  {\rm(i)} of Theorem \ref{ThmIdealOclosure}, there exists an increasing net $(b_\gamma)_{\gamma \in \Gamma}$ in 
$ \mathscr H_ L[A]$ with supremum $x$ (in $ L$). For every $\gamma$ there exists 
$a_{\gamma}\in [A]$, and so in $A$ by assumption,  such that $b_\gamma\le a_{\gamma}$. 
Let $c\in L$ be any upper bound of $\{x\wedge a:a\in A\}$. Then $c\ge x\wedge a_\gamma\ge x\wedge b_\gamma$ for every $\gamma$. By JID,
\[x=x\wedge\bigvee_\gamma^L b_\gamma=\bigvee_\gamma^L(x\wedge b_\gamma)\leq c.\]
Thus, $x$ is the least upper bound of $\{x\wedge a:a\in A\}$, and hence $x=\bigvee_{a\in A}^L x\wedge a$.  
\end{proof}

Given subsets $A$ and $B$ of the lattice $L$, let $A\vee B:=\{a\vee b:a\in A,\,b\in B\}$ and $A\wedge B:=\{a\wedge b: a\in A,\,b\in B\}$. If $A$ and $B$ are ideals, then $A\wedge B=A\cap B$ is an ideal.  If, in addition, $L$ is distributive, then $A\vee B$ is an ideal as well.

\begin{proposition}\label{p3}
Let $A$ and $B$ be ideals in a lattice $L$.
\begin{enumerate}[label={\rm(\roman*)}]
\item $A^{\sss O}_1\,\wedge\, B^{\sss O}_1\,\supseteq\,(A\wedge B)^{\sss O}_1$ and if $L$ satisfies \textup{(JID)},  then $A^{\sss O}_1\,\wedge\, B^{\sss O}_1\,=\,(A\wedge B)^{\sss O}_1$.
    \item  $A^{\sss O}_1\,\vee\, B^{\sss O}_1\,\subseteq\,(A\vee B)^{\sss O}_1$ and if $L$ is  
    distributive and Dedekind complete, 
    then $A^{\sss O}_1\,\vee\, B^{\sss O}_1\,=\,(A\vee B)^{\sss O}_1$.
\end{enumerate}
\end{proposition}
\begin{proof}
{\rm(i)}~If $x\in (A\cap B)^{\sss O}_1$ then $x$ is equal to the supremum of an increasing net in $A\cap B$ by Theorem \ref{ThmIdealOclosure} {\rm(i)}.  Thus, $x\in A^{\sss O}_1\cap B^{\sss O}_1\subseteq A^{\sss O}_1\,\wedge\, B^{\sss O}_1$.  If $L$ satisfies \textup{(JID)}, by {\rm(ii)} of the same theorem,  $A^{\sss O}_1$ and $B^{\sss O}_1$ are ideals, and therefore $A^{\sss O}_1\wedge B^{\sss O}_1=A^{\sss O}_1\,\cap\, B^{\sss O}_1$. Thus, if $x\in A^{\sss O}_1\,\wedge\, B^{\sss O}_1$, then, using Proposition \ref{3.4},  one has $x=\bigvee_{a\in A}x\wedge a=\bigvee_{b\in B}x\wedge b$.  Then
\[x=x\wedge x\ =\ \bigvee_{a\in A}x\wedge a\ \ \wedge\ \ \bigvee_{b\in B}x\wedge b\ =\ \bigvee_{a\in A,\,b\in B}x\wedge a\wedge b\,,\]
and therefore $x\in (A\wedge B)^{\sss O}_1=(A\cap B)^{\sss O}_1$.

{\rm(ii)}~Let $a\in A^{\sss O}_1$ and $b\in B^{\sss O}_1$. Then there are 
nets $(a_\gamma)_{\gamma\in\Gamma} \subseteq A$ and 
$(b_\omega)_{\omega\in\Omega} \subseteq B$, increasing to $a$ and $b$, respectively.  
Then $(a_\gamma\vee b_\omega)_{(\gamma,\omega)\in \Gamma\times\Omega}$ is a net in 
$A\vee B$ that increases to $a\vee b$. This shows that 
$A^{\sss O}_1\,\vee\,B^{\sss O}_1\subseteq(A\vee B)^{\sss O}_1$.  
Assume now that $L$ is distributive and Dedekind complete, and let 
$x\in (A\vee B)^{\sss O}_1$.  The distributivity of $L$ implies that $A\vee B$ is an ideal, 
and therefore by Theorem \ref{ThmIdealOclosure} there exists an increasing net 
$(a_\gamma\vee b_\gamma)_{\gamma\in\Gamma}$ in $A\vee B$ with supremum $x$. 
Dedekind completeness gives $a:=\bigvee_\gamma a_\gamma$ and $b:=\bigvee_\gamma b_\gamma$. Then $x=a\vee b$. The component nets need not be increasing, so let $\mathcal F(\Gamma)$ be the directed set of finite nonempty subsets of $\Gamma$ and put $a_F:=\bigvee_{\gamma\in F}a_\gamma$. Since $A$ is an ideal, $a_F\in A$, and $a_F\uparrow a$. Hence $a\in A_1^O$. Similarly $b\in B_1^O$. Therefore $x=a\vee b\in A_1^O\vee B_1^O$. 
\end{proof}

\begin{corollary}\label{c16} Let $ A$, $B$ be ideals in a lattice $L$ that satisfies \textup{(JID)}.   
\begin{enumerate}[label={\rm(\roman*)}]
\item $\Ocl{A}\cap\Ocl{B}=\Ocl{A\cap B}$.
\item If $L$ is also Dedekind complete, then $\Ocl{A}\vee\Ocl{B}=\Ocl{A\vee B}$.
\end{enumerate}    
\end{corollary}

Under \textup{(JID)}, since \(\mathfrak I_L=K(L)\), several structural features of
\(\mathfrak I_L\) are already implicit in the classical theory of Crawley's
completion.  In particular, \(K(L)\) is a complete lattice and, as recalled by Cornish \cite{Cornish1974}, Crawley's
results imply that it is distributive whenever \(L\) is distributive and
conditionally upper continuous; moreover, Cornish
\cite[Corollary~2.2.2]{Cornish1974} shows that \(K(L)\) satisfies JID
whenever \(L\) does.  The following theorem records these properties in the
present order-convergence framework and, importantly for what follows,
expresses the arbitrary lattice operations explicitly in terms of the
\(\sigma_L\)-closure.

\begin{theorem}\label{thm:IL}
    Let $L$ be a lattice satisfying \textup{(JID)}.  The collection $\mathfrak I_ L$ of all \(O\)-closed ideals of $ L$ forms a complete  lattice satisfying \textup{(JID)}\footnote{Recall that a complete lattice satisfying the
join-infinite distributive law (JID) is called
a \emph{frame} (see, e.g., \cite{GierzEtAl2003}).
Dually, a complete lattice satisfying the
meet-infinite distributive law (MID) is called
a \emph{coframe}.} when ordered by set inclusion.  
    For any family $\{A_i:i\in I\}$ in $\mathfrak I_ L$, we have,
    \begin{equation}\label{re1} 
    \bigvee_i{}^{\!\mathfrak I_ L} A_i=\left(\bigcup_i A_i\right)^{\sigma_ L}\quad\text{and}\quad \bigwedge_i{}^{\!\mathfrak I_ L}A_i=\bigcap_i A_i.
    \end{equation}
    \end{theorem}
    \begin{proof}
     The operation $A\mapsto A^{\sigma_ L}$ ($A\subseteq L$) is a closure operator 
     (see \cite[p. 111]{Birkhoff}).  The subsets of $ L$ satisfying $A=A^{\sigma_ L}$ are precisely the 
     \(O\)-closed ideals.  By \cite[Theorem 2, p. 112]{Birkhoff}, $\mathfrak I_ L$ is a complete lattice 
     under set inclusion, and the equations (\ref{re1}) are satisfied.  
     

Let us show that $\mathfrak I_ L$ satisfies \textup{(JID)}.  
To this end let $A_i\in\mathfrak I_ L$ for every $i\in I$, and let $A\in\mathfrak I_ L$.  
We first claim that
     \[\mathscr H_L\!\biggl[\bigcup_{i\in I}A_i\biggr]\cap A
     =\mathscr H_L\!\biggl[\bigcup_{i\in I}(A_i\cap A)\biggr].\]
     The inclusion from right to left is immediate. We show the converse.  If $\bigcup_{i\in I}A_i=\emptyset$, both sides of the claimed
identity equal $\mathscr H_L[\emptyset]$, since this is the
least ideal of $L$ and is therefore contained in $A$.
We may thus assume that $\bigcup_{i\in I}A_i\ne\emptyset$.  If
     \(x\in\mathscr H_L[\bigcup_iA_i]\cap A\), then for some
     \(a_1,\ldots,a_n\), with \(a_k\in A_{i_k}\),
     \(x\leq a_1\vee\cdots\vee a_n\). Since \(A\) is a downset,
     \(x\wedge a_k\in A\cap A_{i_k}\), and distributivity gives
     \[
       x=x\wedge(a_1\vee\cdots\vee a_n)
       =(x\wedge a_1)\vee\cdots\vee(x\wedge a_n).
     \]
     This proves the claim, and therefore
     \[\biggl(\,\bigcup_{i\in I}A_i\biggr)^{\sigma_ L}\cap A=
     \biggl(\,\bigcup_{i\in I}(A_i\cap A)\biggr)^{\sigma_ L}.\]
     by Corollary \ref{c16}~{\rm(i)}. This shows that $\mathfrak I_ L$ satisfies 
     \textup{(JID)}. 
\end{proof}

\subsection{Homomorphisms induced on \(O\)-closed ideals}

\begin{proposition}\label{pr2}
Let $ L$ be a lattice satisfying \textup{(JID)}, and let $Y$ be a join-regular sublattice of $ L$. Then for every $\emptyset\neq A\subseteq Y$,
\begin{enumerate}[label={\rm(\roman*)}]
\item $A^{\sigma_Y}=A^{\sigma_ L}\cap Y$, and
    \item $A^{\sigma_Y\sigma_{ L}}=A^{\sigma_{ L}}$,    
\end{enumerate}
\end{proposition}

\begin{proof}
{\rm(i)}~ Clearly, if \(Y\) is join regular in $L$, and \(L\) satisfies \textup{(JID)}, then \(Y\) satisfies \textup{(JID)}. For $x\in Y$,  by {\rm (ii)} of Proposition \ref{3.4},
\[
x\in A^{\sigma_Y}\Leftrightarrow x\in \bigl([A]\bigr)^{\sigma_Y}\Leftrightarrow x=\bigvee_{s\in [A]}{\!\!}^Y (x\wedge s)=\bigvee_{s\in [A]}{\!\!}^ L (x\wedge s)\Leftrightarrow x\in\bigl([A]\bigr)^{\sigma_ L}\Leftrightarrow x\in A^{\sigma_ L}.
\]

{\rm(ii)}~We have $A^{\sigma_Y}\subseteq A^{\sigma_ L}$ by {\rm(i)}.  Therefore $A^{\sigma_Y\sigma_{ L}}\subseteq A^{\sigma_{ L}\,\sigma_ L}=A^{\sigma_ L}$.  Conversely, $A\subseteq A^{\sigma_Y}$ and therefore $A^{\sigma_ L}\subseteq A^{\sigma_Y\sigma_{ L}}$.
\end{proof}

The purpose of the lattice $\mathfrak I_L$ is that it provides a completion-like object 
on which lattice homomorphisms extend naturally. The next results show that every 
lattice homomorphism induces a canonical morphism between the corresponding lattices of \(O\)-closed ideals. 
This construction will later be transferred to Dedekind--MacNeille completions.

\begin{lemma}\label{l1} 
Let $K$ and $L$ be lattices satisfying \textup{(JID)}.  
Let $T: K\to  L$ be a lattice homomorphism preserving every nonempty join that exists in its domain. 
Then, $T(A^{\sigma_{K}})\subseteq (T(A))^{\sigma_ L}$ for every  $\emptyset\ne A\subseteq K$.    
\end{lemma}
\begin{proof}
First, suppose that $A$ is a sublattice of $K$. Note that if $A$ is a sublattice of 
$K$, then $T(A)$ is a sublattice of $ L$ and therefore we can apply 
Proposition \ref{3.4}~{\rm(ii)}: if $x\in A^{\sigma_K}$, then,
\[T(x)=T\left(\bigvee_{a\in A}{}^{\!K}\,x\wedge a\right)=
\bigvee_{a\in A}{}^{\! L}\,\,T(x)\wedge T(a),\]
 and therefore,  $T(x)\in (T(A))^{\sigma_ L}$.

When $A\subseteq K$ is just a nonempty subset, observe that
$[A]^{\sigma_K}=A^{\sigma_K}$, and 
$(T([A]))^{\sigma_ L}=\bigl([T(A)]\bigr)^{\sigma_ L}=(T(A))^{\sigma_ L}$.
Therefore, $T(A^{\sigma_K})=T\left([A]^{\sigma_{K}}\right)\subseteq (T([A]))^{\sigma_{ L}}=
(T(A))^{\sigma_ L}$. 
\end{proof}

The preceding results allow lattice homomorphisms to induce canonical maps
between the corresponding lattices of \(O\)-closed ideals.

\begin{theorem}\label{tr4}
    Let $K$ and $ L$ be lattices satisfying \textup{(JID)} and let $T:K\to  L$ be a lattice homomorphism.   Define
    \[\hat T:\mathfrak I_{K}\to\mathfrak I_ L:A\mapsto (T(A))^{\sigma_ L}.\]
   \begin{enumerate}[label={\rm(\roman*)}]
    \item $\hat T$ is a meet homomorphism.
    \item If $T$ preserves every nonempty join that exists in its domain, then so does $\hat T$.
    \end{enumerate}
\end{theorem}
\begin{proof}
{\rm(i)}~If $A=\emptyset$ or $B=\emptyset$, the required equality follows
because $\hat T(\emptyset)=\emptyset^{\sigma_L}$ is the least
element of $\mathfrak I_L$. We may therefore assume that
$A,B\ne\emptyset$.  We show that 
$\hat T(A\,\wedge^{\mathfrak I_{K}}\, B)=\hat T(A)\,\wedge^{\mathfrak I_ L}\,\hat T(B)$ for every $A, B\in \mathfrak I_{K}$.  On one-hand we observe that \[\hat T(A\,\wedge^{\mathfrak I_{K}}\, B)=\bigl(T(A\cap B)\bigr)^{\sigma_ L}\subseteq \bigl(T(A)\bigr)^{\sigma_ L}\cap \bigl(T(B)\bigr)^{\sigma_ L}=\hat T(A)\,\wedge^{\mathfrak I_ L}\,\hat T(B).\]  
For the reverse inclusion, we make use of Proposition \ref{3.4} {\rm(ii)}. Observe that $T(A)$ and $T(B)$ are sublattices of $ L$, and therefore, if $x\in \bigl(T(A)\bigr)^{\sigma_ L}\cap \bigl(T(B)\bigr)^{\sigma_ L}$ then $x=\bigvee_{a\in A}^ L \,x\wedge T(a)$ and $x=\bigvee_{b\in B}^ L \,x\wedge T(b)$.  Therefore, by \textup{(JID)}, we get 
\[T(a)\wedge x = T(a)\,\wedge\left(\bigvee_{b\in B}{\!}^ L \,x\wedge T(b)\right) =\bigvee_{b\in B}{\!}^ L \,\,T(a)\wedge x\wedge T(b)\,.\]
Thus, if $c\in  L$ satisfies $c\ge x\wedge T(a)\wedge T(b)$ for every $a\in A$, $b\in B$, then $c\ge x\wedge T(a)$ for every $a\in A$, and therefore $c\ge \bigvee_{a\in A}^ L\,x\wedge T(a) =x$.  This shows that 
\[\bigvee_{a\in A,\, b\in B}^ L x\wedge T(a\wedge b)=\bigvee_{a\in A,\, b\in B}^ L x\wedge T(a)\wedge T(b)=x.\]  Since $A\cap B=\{a\wedge b:a\in A,\, b\in B\}$, Proposition \ref{3.4} {\rm(ii)} yields $x\in \bigl(T(A\cap B)\bigr)^{\sigma_ L}=\hat T(A\,\wedge^{\mathfrak I_{K}}B)$.

{\rm(ii)}~Now we suppose that $T$ preserves arbitrary nonempty joins that exist in $K$.  Let us show that $\hat T\left(\bigvee_i^{\mathfrak I_K} A_i\right)=\bigvee_i^{\mathfrak I_ L}\hat T(A_i)$ for every nonempty family $\{A_i:i\in I\}\subseteq \mathfrak I_K$. 
The function $\hat T$ is isotone, and so $\hat T\left(\bigvee_i^{\mathfrak I_K} A_i\right)$ is an upper-bound of $\{\hat T(A_i):i\in I\}$. Therefore, $\hat T\left(\bigvee_i^{\mathfrak I_K} A_i\right)\supseteq \bigvee_i^{\mathfrak I_ L}\hat T(A_i)$.  

For the reverse inclusion, we make use of Lemma \ref{l1}.  If $\bigcup_{i\in I}A_i=\emptyset$, then every $A_i$
is empty and the required equality is immediate.
We may therefore assume that $\bigcup_{i\in I}A_i\ne\emptyset$.  Then, by Lemma \ref{l1},
\[
\hat T \left(\bigvee_{i\in I}{\!}^{{\mathfrak I_{K}}} A_i\right)=\left(T\biggl(\,\biggl(\bigcup_{i\in I} A_i\biggr)^{\sigma_{K}}\,\biggr)\right)^{\sigma_ L}\subseteq\left(\bigcup_{i\in I} T(A_i)\right)^{\sigma_ L}\subseteq \left(\bigcup_{i\in I} \bigl(T(A_i)\bigr)^{\sigma_ L}\right)^{\sigma_ L}=\bigvee_{i\in I}{\!}^{\mathfrak I_ L}\,\,\hat T(A_i).
\]
\end{proof}

Part~{\rm(ii)} of Theorem~\ref{tr4} is an instance of
\cite[Theorem~3.1]{Cornish1974}, applied to
$\kappa_L\circ T:K\to K(L)$, where $\kappa_L(y)=\downarrow y$.
The resulting extension is precisely
$A\mapsto(T(A))^{\sigma_L}$.\footnote{When the source has a
least element, apply Cornish's theorem in the target interval
above its image, to accommodate our convention of preserving
only nonempty joins.}
Our proof also establishes part~{\rm(i)} without assuming
preservation of existing joins.

\subsection{Comparison with the Dedekind--MacNeille completion}

For a subset \(A\) of a lattice \(L\), let \(A^+\) and \(A^-\) denote the sets
of upper and lower bounds of \(A\), respectively. Recall that a subset
\(C\subseteq L\) is called a \emph{lower cut} if \(C=C^{+-}\).
The collection of all lower cuts of \(L\), ordered by inclusion, which we denote by \(\DM(L)\), will be identified with the 
Dedekind--MacNeille completion of \(L\).  The
canonical embedding is given by $x\longmapsto \downarrow x$, where $\downarrow x:=\{x\}^-$ 
The following theorem is classical and will be used repeatedly throughout the
paper.

\begin{theorem}[Dedekind--MacNeille completion]
Let \(L\) be a lattice. Then \(\DM(L)\) is a complete lattice. Moreover,
the map 
\[
\varphi:L\to\DM(L)\,:\,x\longmapsto\, \downarrow x
\]
is a regular lattice-embedding, and the image of \(L\) is
both join-dense and meet-dense in \(\DM(L)\). In particular, 
   \[ a =\bigvee{}^{\!{ \DM(L)}}\,\set{\varphi(x)}{x\in L,\,\varphi(x)\le  a}\,,\]
   and 
   \[ a =\, \bigwedge{}^{\! \DM(L)}\set{\varphi(x)}{x\in L,\,\varphi(x)\ge  a}\,,\]
   for every $ a\in \DM(L)$.
\end{theorem}
\begin{proof}
See, for example, \cite{MacNeille1937}, \cite[Chapter~X]{Birkhoff}, or
\cite[Section~7.4]{DaveyPriestley}. The completeness of \(\DM(L)\) follows
from the fact that if $\set{ D_i}{i\in I}\subseteq \DM (L)$  then 
 \[\bigvee_{i\in I}{\!}^{\DM (L)}D_i=\left(\bigcup_{i\in I} D_i\right)^{+-}\quad\text{and}
 \quad\bigwedge_{i\in I}{\!}^{\DM (L)}D_{i}=\bigcap_{i\in I} D_{i}\,.\]
The remaining assertions are
the standard density properties of the MacNeille embedding.
\end{proof}



\begin{theorem}[\cite{AbelaChetcutiPositivity}, Theorem 20]\label{28}
	Let $Y$ be a sublattice of a lattice $L$.  Let $\iota:Y\hookrightarrow L$ be the inclusion map.  The map  
	\[\iota_{\DM}:\DM(Y)\to\DM(L): A\mapsto A^{+-} \]
	is an order-embedding\footnote{i.e. $i(A)\subseteq  i(B)$ iff $A\subseteq B$ for every $A$, $B$ in $\DM(Y)$} of $\DM(Y)$ into $\DM(L)$.  
\end{theorem}	

We shall refer to $\iota_{\DM}$ as the \emph{canonical extension of the inclusion map $\iota:Y\hookrightarrow L$}.



The equivalence of {\rm(i)} and {\rm(ii)} in the following
proposition is the lattice specialization of
\cite[Theorem~13]{Niederle2005}. The equivalence of {\rm(i)} and
{\rm(iii)} is Cornish's characterization
\cite[Corollary~3.2.2]{Cornish1974}, expressed through
$\mathfrak I_L=K(L)$.  We include a direct proof
connecting the cut and $O$-closed-ideal formulations; the latter
provides the identification used in Theorem~\ref{thmextension}.

\begin{proposition}\label{macneille}
Let \(L\) be a lattice. For \(a,b\in L\), put
\[
S(a,b):=\{x\in L:x\wedge a\leq b\}.
\]
Then the following conditions are equivalent:
\begin{enumerate}[label={\rm(\roman*)}]
\item The Dedekind--MacNeille completion \(\DM(L)\) satisfies \textup{(JID)}.
\item $S(a,b)$ is a lower cut for every \(a,b\in L\).
\item The lattice \(L\) satisfies \textup{(JID)} and every \(O\)-closed
ideal of \(L\) is a lower cut.
\end{enumerate}
\end{proposition}
\begin{proof}

\noindent{\rm(i)}$\Rightarrow${\rm(ii)}  Assume that \(\DM(L)\) satisfies \textup{(JID)}.  Fix $a,b\in L$.  Since 
\[
S(a,b)^{+-}
=
\bigveeDML
\set{\down y}{y\in S(a,b)},
\]
the \textup{(JID)} property of \(\DM(L)\) gives
\begin{align*}
\downarrow a\,\cap\, S(a,b)^{+-}
=&
\down a\,\cap\,
\bigveeDML
\set{\down y}{y\in S(a,b)}\\
=&
\bigveeDML
\set{\down a\,\cap\,\down y}{y\in S(a,b)}\\
=&
\bigveeDML
\set{\down(a\wedge y)}{y\in S(a,b)}\\
\subseteq&\down b.
\end{align*}

Therefore, for \(x\in S(a,b)^{+-}\) one has 
\[x\wedge a\,\in\,\down a\,\cap\, S(a,b)^{+-}\,\subseteq\,\downarrow b,\]
i.e. $x\in S(a,b)$. This shows that \(S(a,b)^{+-}\subseteq S(a,b).\)
The reverse inclusion \(S(a,b)\subseteq S(a,b)^{+-}\) holds a fortiori, and consequently, $S(a,b)^{+-}=S(a,b)$.

\medskip
\noindent
{\rm(ii)}\(\Rightarrow\){\rm(iii)}
We first prove that \(L\) satisfies \textup{(JID)}. Let
\(A\subseteq L\), suppose that \(s=\bigveeL A\) exists in \(L\), and
fix \(x\in L\). We show that
\[
x\wedge s=\bigveeL_{a\in A}(x\wedge a).
\]
The element \(x\wedge s\) is clearly an upper bound of
\(\{x\wedge a:a\in A\}\). Let \(c\in L\) satisfy $x\wedge a\leq c$ for every $a\in A$.  Then $A\subseteq S(x,c)$. By assumption, \(S(x,c)\) is a lower cut. Therefore $A^{+-}\subseteq S(x,c)$.
Since \(s=\bigveeL A\), the element \(s\) is below every upper bound
of \(A\), and hence \(s\in A^{+-}\). Thus \(s\in S(x,c)\), so
\[
x\wedge s\leq c.
\]
It follows that
\[
x\wedge\bigveeL A
=
\bigveeL_{a\in A}(a\wedge x).
\]
Therefore \(L\) satisfies \textup{(JID)}.


Now let \(I\) be an \(O\)-closed ideal of \(L\). We prove that
\(I=I^{+-}\). The inclusion \(I\subseteq I^{+-}\) is automatic. Let
\(x\in I^{+-}\). We claim that
\[
x=\bigveeL_{y\in I}(x\wedge y).
\]
Certainly, \(x\) is an upper bound of
\(\{x\wedge y:y\in I\}\). Let \(c\in L\) be any upper bound of this
set. Then $x\wedge y\leq c$ for every $y\in I$,
and hence $I\subseteq S(x,c)$.  
Since \(S(x,c)\) is a lower cut, it follows that $I^{+-}\subseteq S(x,c)$.
In particular, \(x\in S(x,c)\), and therefore $x=x\wedge x\leq c$.  
This proves that
\[
x=\bigveeL_{y\in I}(x\wedge y).
\]
By the characterization of the \(\sigma_L\)-closure established
in Proposition \ref{3.4}, $x\in I^{\sigma_L}$.
Since \(I\) is an \(O\)-closed ideal, $I^{\sigma_L}=I$.
Hence \(x\in I\), and therefore \(I=I^{+-}\).


\medskip
\noindent
{\rm(iii)}\(\Rightarrow\){\rm(i)} Since \(L\) satisfies \textup{(JID)},  Theorem \ref{thm:IL} shows that \(\mathfrak I_L\) satisfies \textup{(JID)}.

Observe that if $C=\bigcap_{u\in C^+}\downarrow u$ for every lower cut $C\subseteq L$.  Since each principal ideal \(\downarrow u\) is \(O\)-closed, it follows that every lower cut of $L$ is \(O\)-closed, i.e. $\DM(L)\subseteq\mathfrak I_L$. On the other-hand, by assumption, every \(O\)-closed ideal is a lower cut.
Thus $\mathfrak I_L=\DM(L)$.  Since both partial orders are given by set  inclusion, and \(\mathfrak I_L\) satisfies \textup{(JID)}, so does \(\DM(L)\).


\end{proof}

\begin{corollary}\label{IL=DM}
Let $L$ be a lattice such that $\DM(L)$ satisfies \textup{(JID)}.  Then $\mathfrak{I}_L=\DM(L)$.
\end{corollary}



Proposition~\ref{macneille} allows the construction of
Theorem~\ref{tr4} to be transferred to the Dedekind--MacNeille
completion. Indeed, when \(\DM(Y)\) satisfies \textup{(JID)}, we may
identify \(\DM(Y)\) with \(\mathfrak I_Y\), and hence use the map induced
on \(O\)-closed ideals to extend a lattice homomorphism \(T:Y\to L\).
This yields the following extension theorem.  It shows that, under this additional assumption, lattice homomorphisms extend from a lattice to its Dedekind--MacNeille completion in a manner analogous to classical extension results for Dedekind complete Riesz spaces.  (See \cite[Theorem 3.1]{LuxembergSchep1979} or \cite[Theorem 3.1]{Bernau1986}.)   There is, however, an important distinction. In the present setting, the extension obtained on \(\DM(Y)\) is automatically a meet-homomorphism (i.e. it preserves nonempty finite meets), but preservation of finite joins  by the extension is guaranteed when the original map \(T\) to preserve arbitrary joins already in \(Y\).  Thus, unlike the classical Riesz-space situation, the finite lattice structure alone seem  not to be sufficient to guarantee preservation of the corresponding lattice operations after extension, even when $\DM(Y)$ is assumed to satisfy \textup{(JID)}.

\begin{theorem}\label{thmextension}  Let $Y$ be a  lattice and let $T:Y\to L$ be a lattice homomorphism, where $L$ is complete lattice satisfying \textup{(JID)}.  Define
\[\overline T:\DM(Y)\to L:x\mapsto\bigvee{}^{\! L}\,\{T(y):y\in Y,\,y\le x\}.\]
\begin{enumerate}[label={\rm(\roman*)}]
\item $\overline T$ is a meet-homomorphism extending $T$.
\item If $\DM(Y)$ satisfies \textup{(JID)} and $T$ preserves every nonempty join that exists in its domain, then so does $\overline T$. 
\end{enumerate}
\end{theorem}
\begin{proof}
We identify $Y$ with its canonical image in $\DM(Y)$. For
$x\in\DM(Y)$, put  $A_x:=\{y\in Y:y\le x\}$.
The map $\overline T$ is isotone. Moreover, if $y\in Y$, then
$T(y)$ is the greatest element of $\{T(a):a\in Y,\ a\le y\}$,
and hence $\overline T(y)=T(y)$.

\smallskip
\noindent\textup{(i)}
Let $x,z\in\DM(Y)$. Applying \textup{(JID)} in $L$ twice and
using the fact that $T$ preserves finite meets, we obtain
\[\overline T(x)\wedge\overline T(z)=\bigveeL_{\substack{a\in A_x\\ b\in A_z}}
      \bigl(T(a)\wedge T(b)\bigr)=\bigveeL_{\substack{a\in A_x\\ b\in A_z}}T(a\wedge b).\]
Since $Y$ is a lattice,
\[
    \{a\wedge b:a\in A_x,\ b\in A_z\}=A_{x\wedge z}.
\]
Consequently,
\[
    \overline T(x)\wedge\overline T(z)
    =\bigvee_{u\in A_{x\wedge z}}^{L}T(u)
    =\overline T(x\wedge z).
\]

\smallskip
\noindent\textup{(ii)}
Suppose that $\DM(Y)$ satisfies \textup{(JID)} and that $T$
preserves every existing nonempty join in $Y$.
Corollary  \ref{IL=DM} yields $\DM(Y)=\mathfrak I_Y$.
Theorem~3.8 therefore applies to the induced map
\[
    \widehat T:\mathfrak I_Y\longrightarrow\mathfrak I_L,
    \qquad A\longmapsto\bigl(T(A)\bigr)^{\sigma_L},
\]
and shows that $\widehat T$ preserves all existing nonempty joins of $\mathfrak I_Y$.  Since $L$ is complete and satisfies \textup{(JID)}, every $O$-closed ideal of $L$ is  principal. Hence the map $j:\ell\longmapsto \downarrow\ell$
defines an order-isomorphism of $L$ onto $\mathfrak I_L$. 
Therefore
\[
    j\circ\overline T(x)\, =\,\downarrow\!\left(\bigveeL_{y\in A_x}T(y)\right)=\bigl(T(A_x)\bigr)^{\sigma_L}=\widehat T(A_x) \quad (x\in\DM(Y)).\]

Thus $\overline T=j^{-1}\circ\widehat T$.  Since $j^{-1}$  and $\widehat T$ preserve arbitrary nonempty joins, so does $\overline T$.
\end{proof}

We can thus obtain the following permanence result for a join-regular sublattice.  This is to be compared to \cite[Theorem 2.10]{GaoTroitskyXanthos}.

\begin{theorem}[Regular MacNeille extension]\label{thm:joinregular}
Let \(Y\) be a join-regular sublattice of a complete lattice \(L\) satisfying \textup{(JID)}. Let \(\iota:Y\hookrightarrow L\) be the inclusion map. If \(\DM(Y)\) satisfies \textup{(JID)},  the canonical extension of \(\iota\) is a lattice embedding of \(\DM(Y)\) into \(L\), and its image is join-regular in \(L\).
\end{theorem}

\begin{proof}
Since \(Y\) is join-regular in \(L\), \(\iota\)  is a lattice-homomorphism that preserves all existing nonempty joins of \(Y\). By Theorem~\ref{thmextension}, \(\iota\) extends to a join-preserving lattice homomorphism
\[
\overline\iota:\DM(Y)\to L:A\mapsto \bigveev{L}\{\iota(a):a\in A\}.
\]
Since \(L\) is complete, the map $x\to \downarrow L$ establishes an order isomorphism between $L$ and $\DM(L)$.  Thus, if we identify $L$ with $\DM(L)$, the canonical extension of $\iota$ can be identified with \(\overline\iota\).  By Theorem \ref{28},  the canonical extension of $\iota$ is an order isomorphism.  In particular it is injective.  Thus, the canonical extension of \(\iota\) is a lattice embedding of \(\DM(Y)\) into \(L\), and its image is join-regular in \(L\).
\end{proof}

For Archimedean vector lattices and $\ell$-groups, regular inclusions
extend to regular embeddings of the corresponding Dedekind completions.
For general lattices, the join-regularity conclusion of
Theorem~\ref{thm:joinregular} is best possible: the following example
shows that the canonical embedding need not be meet-regular, even when
both the ambient lattice and $\DM(Y)$ are complete chains.

\begin{example}[Failure of meet-regularity]
\label{ex:failure-meet-regularity}
Let
\[
    L=[-1,2],\qquad Y=[-1,0)\cup(1,2],
\]
with their usual orders. The inclusion $\iota:Y\hookrightarrow L$
is regular. Indeed, the supremum in $L$ of a nonempty subset
$S\subseteq Y$ either belongs to $Y$ or equals $0$; in the latter
case, its upper bounds in $Y$ form $(1,2]$, which has no least
element. Thus every existing nonempty join in $Y$ is preserved.
Meet-regularity follows by duality, since $t\mapsto1-t$ is an
order-reversing involution of $L$ preserving $Y$.

The only nonprincipal lower cut is $D=[-1,0)$: indeed,
$D^+=(1,2]$ and $D^{+-}=D$, while every other lower cut has its
supremum in $L$ belonging to $Y$, and is therefore principal.
Since $Y$ is a chain, $\DM(Y)$ is a complete chain and satisfies
both \textup{(JID)} and \textup{(MID)}.

The canonical extension $\overline\iota(C)=\bigvee{}^{\!L}C$
satisfies
\[
    \overline\iota(\downarrow y)=y\quad(y\in Y),
    \qquad \overline\iota(D)=0.
\]
By the preceding theorem, it is a lattice embedding with
join-regular image $H=[-1,0]\cup(1,2]$.
However, for $y_n=1+1/n$, $n\ge1$, one has
\[
    \bigwedge_{n\ge1}^{\DM(Y)}\downarrow y_n
    =\bigcap_{n\ge1}\downarrow y_n=D,
\]
whereas
\[
    \overline\iota\!\left(
        \bigwedge_{n\ge1}^{\DM(Y)}\downarrow y_n
    \right)
    =0
    \ne
    1=\bigwedge_{n\ge1}^{L}
        \overline\iota(\downarrow y_n).
\]
Thus the canonical image need not be meet-regular, even when
both $\DM(Y)$ and $L$ are complete chains.
\end{example}

\section{An obstruction to preservation under Dedekind--MacNeille completion}
\label{sec:counterexample}

The hypothesis on $\DM(Y)$ in the preceding results is substantive.  The
following example gives a concrete coordinate version, inside $\RR^3$, of the
classical obstruction behind the work of Dilworth--McLaughlin
\cite{DilworthMcLaughlin}.  It also complements later criteria for
distributivity of Dedekind-MacNeille completions \cite{Erne1982,BallPultrNormalCompletion}.

\subsection{A regular sublattice of $\mathbb R^3$ whose Dedekind--MacNeille completion does not allow for a lattice embedding in $\RR^3$}
\label{subsec:concrete-dm-r3}

For $n,m\in\NN$, put
\[
 a_{nm}=\left(\frac1n,\frac1m,1\right),\ \ b_{nm}=\left(-\frac1n,\frac1m,1\right),\ \ c_{nm}=\left(-\frac1n,\frac1m,-1\right),\ \ d_{nm}=\left(-\frac1n,-\frac1m,-1\right),\]
and let $A,B,C,D$ denote the corresponding four families.  Set $L=A\cup B\cup C\cup D\subseteq\RR^3$.
The missing limit planes $x=0$ and $y=0$ are responsible for the extra
MacNeille cuts appearing below.

\begin{figure}[H]
\centering
\scalebox{0.95}{
\begin{tikzpicture}

\begin{axis}[
 view={122}{20},
 width=.97\textwidth,
 height=.97\textwidth,
 axis lines=center,
 xlabel={$x$},ylabel={$y$},zlabel={$z$},
 xtick={-2,0,2},
 ytick={-2,0,2},
 ztick={-2,0,2},
 grid=major,
 grid style={gray!18},
 tick label style={font=\small},
 label style={font=\small},
 every axis plot/.append style={line width=.1pt},
 ]

\draw[line width=0.3 pt] (-1,0,-1)--(-1,0,1);
\addplot3[patch,patch type=rectangle,draw=Acol!40,fill=Acol!7,
  opacity=.65] coordinates {(0,0,1) (1,0,1) (1,1,1) (0,1,1)};
\addplot3[patch,patch type=rectangle,draw=Bcol!45,fill=Bcol!8,
  opacity=.65] coordinates {(-1,0,1) (0,0,1) (0,1,1) (-1,1,1)};
\addplot3[patch,patch type=rectangle,draw=Ccol!45,fill=Ccol!7,
  opacity=.58] coordinates {(-1,0,-1) (0,0,-1) (0,1,-1) (-1,1,-1)};
\addplot3[patch,patch type=rectangle,draw=Dcol!45,fill=Dcol!7,
  opacity=.58] coordinates {(-1,-1,-1) (0,-1,-1) (0,0,-1) (-1,0,-1)};
\draw[line width=0.3 pt] (0,1,-1)--(0,1,1);
\draw[line width=0.4 pt](0,0,0)--(0,0,-1);
\draw[line width=0.3 pt] (-1,1,-1)--(-1,1,1);
\draw[line width=0.5 pt,->,Acol] (1,1,1)--(1/10,1,1);
\draw[line width=0.5 pt,->,Acol] (1,1,1)--(1,1/10,1);
\draw[line width=0.5 pt,->,Bcol] (-1,1,1)--(-1/10,1,1);
\draw[line width=0.5 pt,->,Bcol] (-1,1,1)--(-1,1/10,1);
\draw[line width=0.5 pt,->,Ccol] (-1,1,-1)--(-1,1/10,-1);
\draw[line width=0.5 pt,->,Ccol] (-1,1,-1)--(-1/10,1,-1);
\draw[line width=0.5 pt,->,Dcol] (-1,-1,-1)--(-1,-1/10,-1);
\draw[line width=0.5 pt,->,Dcol] (-1,-1,-1)--(-1/10,-1,-1);

\foreach \n in {1,...,5}{
  \foreach \m in {1,...,5}{
    \addplot3[only marks,mark=*,mark size=1.25pt,Acol]
      coordinates {(1/\n,1/\m,1)};
    \addplot3[only marks,mark=triangle*,mark size=1.55pt,Bcol]
      coordinates {(-1/\n,1/\m,1)};
    \addplot3[only marks,mark=square*,mark size=1.25pt,Ccol]
      coordinates {(-1/\n,1/\m,-1)};
    \addplot3[only marks,mark=diamond*,mark size=1.45pt,Dcol]
      coordinates {(-1/\n,-1/\m,-1)};
  }
}

\foreach \k in {1,...,5}{
 \addplot3[Acol!48] coordinates
  {(1/\k,1,1) (1/\k,1/2,1) (1/\k,1/3,1) (1/\k,1/4,1)
   (1/\k,1/5,1) };
 \addplot3[Bcol!55] coordinates
  {(-1/\k,1,1) (-1/\k,1/2,1) (-1/\k,1/3,1) (-1/\k,1/4,1)
   (-1/\k,1/5,1) };
 \addplot3[Ccol!50] coordinates
  {(-1/\k,1,-1) (-1/\k,1/2,-1) (-1/\k,1/3,-1)
   (-1/\k,1/4,-1) (-1/\k,1/5,-1)};
 \addplot3[Dcol!50] coordinates
  {(-1/\k,-1,-1) (-1/\k,-1/2,-1) (-1/\k,-1/3,-1)
   (-1/\k,-1/4,-1) (-1/\k,-1/5,-1) };
}

\foreach \k in {1,...,5}{
 \addplot3[Acol!48] coordinates
  {(1,1/\k,1) (1/2,1/\k,1) (1/3,1/\k,1) (1/4,1/\k,1)
   (1/5,1/\k,1)};
 \addplot3[Bcol!55] coordinates
  {(-1,1/\k,1) (-1/2,1/\k,1) (-1/3,1/\k,1) (-1/4,1/\k,1)
   (-1/5,1/\k,1)};
 \addplot3[Ccol!50] coordinates
 {(-1,1/\k,-1) (-1/2,1/\k,-1) (-1/3,1/\k,-1) (-1/4,1/\k,-1)
   (-1/5,1/\k,-1)};
 \addplot3[Dcol!50] coordinates
  {(-1,-1/\k,-1) (-1/2,-1/\k,-1) (-1/3,-1/\k,-1)
   (-1/4,-1/\k,-1) (-1/5,-1/\k,-1)};
}
\addplot3[only marks,mark=*,mark size=1pt,Ccol]
 coordinates {(-1,1,-1)} node[left=0.5 pt] {${ c_{11}}$};
\addplot3[only marks,mark=*,mark size=0.5pt,Bcol]
 coordinates {(-1,1,1)} node[left=0.5 pt] {$b_{11}$};
\addplot3[only marks,mark=*,mark size=0.5pt,Acol]
 coordinates {(1,1,1)} node[right=0.5 pt] {$a_{11}$};
\addplot3[only marks,mark=*,mark size=0.5pt,Dcol]
 coordinates {(-1,-1,-1)} node[left=0.5 pt] {$d_{11}$};
\end{axis}
\end{tikzpicture}}
\caption{The four grids defining $L\subseteq\RR^3$.  Only five points in
each index direction are shown; the pale rectangles are visual guides.}
\label{fig:DM-concrete}
\end{figure}
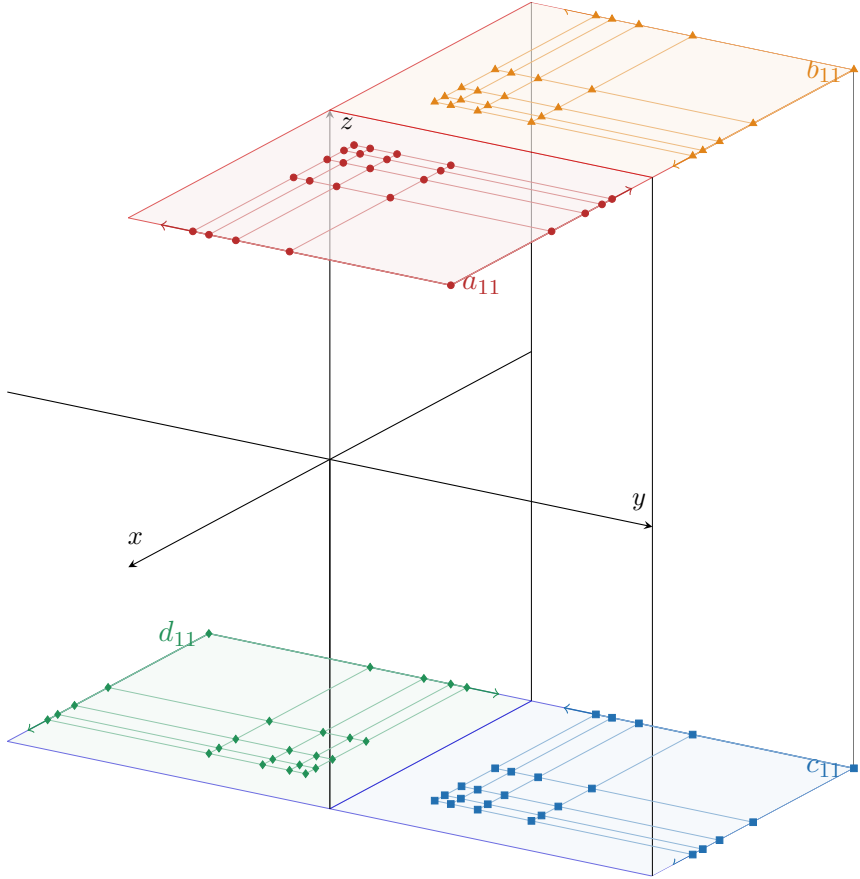

\begin{proposition}\label{prop:R3regular}
$L$ is a bounded regular sublattice of $\RR^3$, with least element $d_{11}$
and greatest element $a_{11}$.  In particular, $L$ is infinitely distributive.
\end{proposition}

\begin{proof}
Closure under coordinatewise minima and maxima is immediate from the four
families.  For example,
\[
 d_{nm}\vee c_{rs}=c_{\max\{n,r\},s},\qquad
 c_{nm}\wedge a_{rs}=c_{n,\max\{m,s\}},
\]
and the other cases are analogous.  The bounds are $d_{11}$ and $a_{11}$.

We verify regularity for suprema; infima then follow from the
order-reversing bijection
\[
\theta(x,y,z)=(-y,-x,-z),
\]
which maps $L$ onto itself. Let $S\subseteq L$ have
$s=\sup_L S$, and let $t=\sup_{\mathbb R^3}S$ coordinatewise.
Then $t\le s$. If $s\in D$, the coordinate suprema are attained,
and hence $t\in D$. If $s\in C$ (respectively, $B$), the same
conclusion holds unless a coordinate supremum is $0$. In that case,
by increasing the corresponding index one obtains an element of
$C$ (respectively, $B$) strictly below $s$ which is still an upper
bound of $S$, contradicting the minimality of $s$. Finally, suppose
$s\in A$. If $S\cap A\ne\varnothing$, the positive coordinate
suprema are attained, so $t\in A$. If $S\cap A=\varnothing$, then
the first coordinate of $t$ is at most $0$, and increasing the first
index of $s=a_{nm}$ produces an element $a_{rm}<s$ which is still
an upper bound of $S$, again a contradiction. Thus $t\in L$ in all
cases, and the minimality of $s$ gives $s=t$.


Since $\RR^3$ is infinitely distributive, so is $L$.
\end{proof}

\begin{proposition}\label{prop:Dcut}
$D$ is a MacNeille cut.
\end{proposition}

\begin{proof}
Every member of $A$ dominates $D$, whereas no element of $B\cup C\cup D$
does, because its first coordinate is a fixed negative number while
$-1/n\uparrow0$.  Thus $D^+=A$.  The elements of $L$ below every member of
$A$ are exactly those in $D$, so $A^-=D$ and $D^{+-}=D$.
\end{proof}

Put
\[
             p=c_{11}=(-1,1,-1),\qquad
             q=b_{11}=(-1,1,1),
\]
so that $p<q$, and identify these elements with their principal lower cuts.  Directly,
\[
 D\cap\downarrow p=D\cap\downarrow q=\{d_{1m}:m\in\NN\},             \tag{1}
\]
and
\[
 (D\cup\downarrow p)^+=(D\cup\downarrow q)^+
      =\{a_{r1}:r\in\NN\}.                                           \tag{2}
\]
The common lower bounds of the family in (2) are precisely $B\cup C\cup D$;
therefore
\[
 D\vee\downarrow p=D\vee\downarrow q=B\cup C\cup D.                \tag{3}
\]
Consequently,
\[
 (D\vee\downarrow p)\wedge\downarrow q=\downarrow q,
 \qquad
 (D\wedge\downarrow q)\vee\downarrow p=\downarrow p.                \tag{4}
\]
Since $\downarrow p<\downarrow q$, (4) violates the modular identity.

\begin{theorem}[Concrete obstruction]\label{thm:R3obstruction}
There exists an infinitely distributive lattice $L$ which is a regular
sublattice of the Dedekind complete vector lattice $\RR^3$, while $\DM(L)$
is not modular.  In particular, $\DM(L)$ is not distributive and admits no
lattice embedding into $\RR^3$.
\end{theorem}

\begin{proof}
Propositions~\ref{prop:R3regular} and \ref{prop:Dcut}, together with
(1)--(4), give the first two assertions.  Every sublattice of $\RR^3$ is
distributive, so a non-distributive lattice cannot embed into it.
\end{proof}

The example isolates a useful mechanism: $D$ has the same intersection with
$\downarrow p$ and $\downarrow q$, while adjoining either $p$ or $q$ to $D$
produces the same set of upper bounds and hence the same Dedekind--MacNeille closure.
This suggests several questions that sit naturally beside the classical
criteria for distributivity of Dedekind--MacNeille completions
\cite{DilworthMcLaughlin,Erne1982,BallPultrNormalCompletion}, Funayama-type
regular embedding results \cite{Crawley1962,BezhanishviliGabelaiaJibladze},
the general theory of MacNeille extensions \cite{TheunissenVenema}, and more
recent work on regular and related completions
\cite{HardingYang,BezhanishviliEtAl2025}.

\subsection{The same obstruction cannot be achieved in the two-dimensional case}

The following theorem gives a sharp dimensional contrast with
Theorem~\ref{thm:R3obstruction}. For bounded sublattices of
$\RR^2$, join-regularity alone ensures a canonical lattice
embedding of the Dedekind--MacNeille completion into the ambient space.
Thus the lattice-embedding conclusion of
Theorem~\ref{thm:joinregular} holds in this setting without
assuming \textup{(JID)} for $\DM(Y)$.

\begin{theorem}[Canonical realization in two coordinates]
\label{thm:R2-positive}\label{thm:R2-realization}
Let $Y$ be a join-regular sublattice of $\RR^2$ with a least and a
 greatest element, and let $\iota:Y\hookrightarrow\RR^2$ be the inclusion.
Then, the map
\[
\bar\iota:\DM(Y)\longrightarrow\RR^2,
\qquad A\longmapsto\bigveev{\RR^2}A,
\]
is a lattice embedding extending $\iota$. Consequently, $\DM(Y)$ is
 distributive.
\end{theorem}

\begin{proof}
All cut bounds are taken in $Y$. Every lower cut $A$ contains the least
 element of $Y$ and is bounded above by its greatest element; hence
 $\bar\iota(A)$ exists and $A^+\ne\emptyset$. Moreover,
\[
A=Y\cap\downarrow\bar\iota(A).
\]
Indeed, if $y\in Y$ and $y\le\bar\iota(A)$, then
 $y\le\bar\iota(A)\le u$ for every $u\in A^+$, so $y\in A^{+-}=A$;
 the other inclusion is immediate. Thus $\bar\iota$ is an order embedding,
 and $\bar\iota(\downarrow y)=y$. Let $0_Y,1_Y$ denote the least and
 greatest elements of $Y$. The interval $I=[0_Y,1_Y]\subseteq\RR^2$
 is complete and satisfies \textup{(JID)}. Theorem~\ref{thmextension}
 {\rm(i)}, applied to $\iota:Y\to I$, shows that $\bar\iota$ preserves
 binary meets, since its lower extension is the same coordinatewise
 supremum map.

For joins, comparable cuts present no difficulty. Otherwise, write
 $\bar\iota(A)=\alpha$ and $\bar\iota(B)=\beta$, interchanging the cuts
 if necessary so that $\alpha_1<\beta_1$ and $\beta_2<\alpha_2$.
Put $D=(A\cup B)^{+-}$ and suppose that some $y\in D$ satisfies
 $y\nleq(\beta_1,\alpha_2)$. Interchanging both the coordinates and the
 cuts if necessary, we may assume $y_1>\beta_1$ while retaining these
 inequalities. Choose $b_0\in B$ with $(b_0)_1>\alpha_1$ and
 $a_0\in A$ with $(a_0)_2>\beta_2$. Since $b_0\notin A$ and
 $a_0\notin B$, the cut identities give $u\in A^+$ and $v\in B^+$
 with $b_0\nleq u$ and $a_0\nleq v$. As $u\ge\alpha$ and
 $v\ge\beta$, it follows that
\[
 u_1<(b_0)_1\le\beta_1,\quad u_2\ge\alpha_2,
 \qquad
 v_1\ge\beta_1,\quad v_2<(a_0)_2\le\alpha_2.
\]
Thus $u\vee v=(v_1,u_2)$ is an upper bound of $A\cup B$, whence
 $y\le(v_1,u_2)$ and $h:=y\vee u=(y_1,u_2)$.

We claim that $h=\bigveev{Y}\{b\vee u:b\in B\}$. Indeed, $h$ is an
 upper bound, since $b_1\le\beta_1<y_1$ and
 $b_2\le\beta_2<\alpha_2\le u_2$ for every $b\in B$.
Any upper bound $z\in Y$ of this nonempty set dominates $u$ and $B$,
 hence $A\cup B$; therefore $y\le z$ and $h\le z$.
Join-regularity now gives
\[
 (y_1,u_2)=h
 =\bigveev{\RR^2}\{b\vee u:b\in B\}
 =\beta\vee u=(\beta_1,u_2),
\]
contrary to $y_1>\beta_1$. Consequently every $y\in D$ satisfies
 $y\le\alpha\vee\beta$. Since $A\cup B\subseteq D$, we obtain
 $\bar\iota(A\vee B)=\alpha\vee\beta$, completing the proof.
\end{proof}

\section*{Acknowledgements}
The authors would like to thank Eugene Bilokopytov for helpful conversations on the topic of the paper.

\end{document}